\documentclass[12pt]{article}
\usepackage{amsmath,amssymb,amsthm}
\numberwithin{equation}{section}
\usepackage{hyperref}
\usepackage{units}
\usepackage{color}
\usepackage[T1]{fontenc}
\usepackage[utf8]{inputenc}
\usepackage{authblk}
\usepackage{bm}
\usepackage{graphicx}

\usepackage{datetime}
\usepackage{color}

\usepackage[toc,page]{appendix}

\newcommand{\Z}{{\mathbb Z}}

\newtheorem{thm}{Theorem}[section]

\newtheorem{lemma}[thm]{Lemma}

\newtheorem{cor}[thm]{Corollary}

\title{Sums of fourth powers of Fibonacci and Lucas numbers\thanks{AMS Classification Numbers : 11B37, 11B39}\vspace{10mm}}

\author[]{Kunle Adegoke \thanks{adegoke00@gmail.com, kunle.adegoke@yandex.com}}

\affil{Department of Physics and Engineering Physics, \mbox{Obafemi Awolowo University}, Ile-Ife, Nigeria}

\begin{document}

\date{}

\maketitle

\begin{abstract}
\noindent We obtain closed-form  expressions for all sums of the  form \mbox{$\sum_{k = 1}^n {F_{mk}{}^4 }$} and \mbox{$\sum_{k = 1}^n {L_{mk}{}^4 }$} and their alternating versions, where $F_i$ and $L_i$ denote Fibonacci and Lucas numbers respectively. Our results complement those of Melham who studied the alternating sums.
\end{abstract}

\section{Introduction}

The Fibonacci numbers, $F_n$, and Lucas numbers, $L_n$, are defined, for \mbox{$n\in\Z$}, as usual, through the recurrence relations \mbox{$F_n=F_{n-1}+F_{n-2}$}, \mbox{$F_0=0$, $F_1=1$} and \mbox{$L_n=L_{n-1}+L_{n-2}$}, $L_0=2$, $L_1=1$, with $F_{-n}=(-1)^{n-1}F_n$ and $L_{-n}=(-1)^nL_n$.

\bigskip

About two decades ago, motivated by the results of Clary and Hemenway~\cite{clary} who obtained factored closed-form expressions for sums of the form \mbox{$\sum_{k = 1}^n {F_{mk}{}^3 }$},  Melham~\cite{melham2000} obtained factored closed-form expressions for alternating sums of the form \mbox{$\sum_{k = 1}^n {(-1)^{k-1}F_{mk}{}^4 }$}.

\bigskip

Since no evaluations were reported in the Melham paper for the non alternating sums, we have attempted to fill that gap in this paper. Our main results are the following, valid for integers $m$ and $n$, with $m$ not equal to zero:
\[
25\sum_{k = 1}^n {F_{mk}{}^4 }  = \frac{{F_{2mn+m} (L_{2mn + m}  + 4( - 1)^{mn - 1} L_m )}}{{F_{2m} }} + 6n+3
\]
and
\[
\sum_{k = 1}^n {L_{mk}{}^4 }  = \frac{{F_{2mn+m} (L_{2mn + m}  + 4( - 1)^{mn} L_m )}}{{F_{2m} }} + 6n-5\,.
\]
We also re-derived the alternating sums, in slightly different but equivalent forms to the results contained in~\cite{melham2000}:
\[
\sum_{k = 1}^n {( - 1)^{k - 1} F_{mk}{}^4 }  = \frac{{F_{mn} F_{mn + m} \left\{ (-1)^{n-1}L_m L_{mn} L_{mn + m}  + (-1)^{n(m-1)}4L_{2m} \right\}}}{{5L_m L_{2m} }}
\]
and
\[
\sum_{k = (1+(-1)^n)/2}^n {( - 1)^{k - 1} L_{mk}{}^4 }  = \frac{{( - 1)^{n - 1}5F_{mn} F_{mn + m} \left\{ L_m L_{mn} L_{mn + m}  + ( - 1)^{nm}4L_{2m} \right\}}}{{L_m L_{2m} }}\,,
\]
valid for all integers $m$ and $n$.
\section{Required identities and preliminary results}

\subsection{Telescoping summation identities}
The following telescoping summation identities are special cases of the more general identities proved in~\cite{adegoke}.
\begin{lemma}\label{finall}
If $f(k)$ is a real sequence and $m$ and $n$ are positive integers, then
\[
\sum_{k = 1}^n {\left[ {f(mk + m) - f(mk)} \right]}  = f(mn+m)  -f(m)\,. 
\]

\end{lemma}

\begin{lemma}\label{finqodd}
If $f(k)$ is a real sequence and $m$ and $n$ are positive integers, then
\[
\begin{split}
&\sum_{k = 1}^n {( - 1)^{k - 1} \left[ {f(mk + m) + f(mk)} \right]} \\
&\quad = ( - 1)^{n-1} f(mn+m)+f(m)\,. 
\end{split}
\]
\end{lemma}
\subsection{First-order Lucas summation identities}
\begin{lemma}\label{thm.hnf079l}
If $m$ and $n$ are integers, then
\[
F_m \sum_{k = 1}^n {( - 1)^{mk - 1} L_{2mk} }  = ( - 1)^{mn - 1} F_{mn} L_{mn + m}\,.
\]

\end{lemma}
\begin{proof}
Setting $v=m$ and $u=2mk$ in the identity
\begin{equation}\label{equ.yb05ue2}
F_{u + v} - (-1)^vF_{u-v}=F_vL_u\,, 
\end{equation}
gives
\begin{equation}\label{equ.gimtkiy}
F_{2mk + m}  - F_{2mk - m}  = F_m L_{2mk}\,,\quad\mbox{$m$ even}\,, 
\end{equation}
and
\begin{equation}\label{equ.nmx8lu8}
F_{2mk + m}  + F_{2mk - m}  = F_m L_{2mk}\,,\quad\mbox{$m$ odd}\,. 
\end{equation}
Using identity~\eqref{equ.gimtkiy} in Lemma~\ref{finall} with $f(k)=F_{2k-m}$, it is established that
\begin{equation}\label{equ.s1xp2nb}
\begin{split}
F_m \sum_{k = 1}^n {L_{2mk} }  &= F_{m + 2mn}  - F_m\\ 
 &= F_{m + mn + mn}  - F_{m + mn - mn}\\ 
 &= F_{mn} L_{mn + m}\,,\quad\mbox{$m$ even}\,, 
\end{split}
\end{equation}
on account of identity~\eqref{equ.yb05ue2}. 

\bigskip

Similarly, using identity~\eqref{equ.nmx8lu8} in Lemma~\ref{finqodd} with $f(k)=F_{2k-m}$, we have
\begin{equation}\label{equ.h7znph9}
\begin{split}
F_m \sum_{k = 1}^n {( - 1)^{k - 1} L_{2mk} }  &= ( - 1)^{n - 1} F_{m + 2mn}  + F_m\\ 
 &= ( - 1)^{n - 1} \left( {F_{m + mn + mn}  - ( - 1)^n F_{m + mn - mn} } \right)\\
 &= ( - 1)^{n - 1} \left( {F_{m + mn + mn}  - ( - 1)^{mn} F_{m + mn - mn} } \right)\,,\quad\mbox{since $m$ is odd}\\
 &= ( - 1)^{n - 1} F_{mn} L_{mn + m}\,,\quad\mbox{$m$ odd}\,. 
\end{split}
\end{equation}
Identities~\eqref{equ.s1xp2nb} and \eqref{equ.h7znph9} combine to give Lemma~\eqref{thm.hnf079l}.
\end{proof}
\begin{lemma}\label{thm.njcpb71}
If $m$ and $n$ are integers, then
\[
L_m \sum_{k = 1}^n {( - 1)^{k(m - 1)} L_{2mk} }  = ( - 1)^{n(m - 1)} L_{2mn+m}  - L_m\,. 
\]

\end{lemma}
\begin{proof}
Setting $v=m$ and $u=2mk$ in the identity
\begin{equation}\label{equ.ldun001}
L_{u + v} + (-1)^vL_{u-v}=L_vL_u\,, 
\end{equation}
gives
\begin{equation}\label{equ.efp7baw}
L_{2mk + m}  - L_{2mk - m}  = L_m L_{2mk}\,,\quad\mbox{$m$ odd}\,, 
\end{equation}
and
\begin{equation}\label{equ.fwf5y6v}
L_{2mk + m}  + L_{2mk - m}  = L_m L_{2mk}\,,\quad\mbox{$m$ even}\,. 
\end{equation}
Using~\eqref{equ.efp7baw} in Lemma~\ref{finall} with $f(k)=L_{2k-m}$, we have
\begin{equation}\label{equ.wrfqa3e}
L_m \sum_{k = 1}^n {L_{2mk} } = L_{m + 2mn}  - L_m\,,\quad\mbox{$m$ odd}\,. 
\end{equation}
Similarly, using~\eqref{equ.fwf5y6v} in Lemma~\ref{finqodd} with $f(k)=L_{2k-m}$, we have
\begin{equation}\label{equ.qvew6g0}
L_m \sum_{k = 1}^n {( - 1)^{k - 1} L_{2mk} }  = ( - 1)^{n - 1} L_{m + 2mn}  + L_m\quad\mbox{$m$ even}\,. 
\end{equation}
Identities~\eqref{equ.wrfqa3e} and \eqref{equ.qvew6g0} combine to give Lemma~\eqref{thm.njcpb71}.
\end{proof}

\section{Main results}
\subsection{Non alternating sums}
\begin{thm}\label{thm.rue8com}
If $m$ is a non-zero integer and $n$ is any integer, then
\[
25\sum_{k = 1}^n {F_{mk}{}^4 }  = \frac{{F_{2mn+m} (L_{2mn + m}  + 4( - 1)^{mn - 1} L_m )}}{{F_{2m} }} + 6n+3\,.
\]
\end{thm}

\begin{proof}
By squaring the identity 
\begin{equation}\label{equ.f9c5ppy}
5F_u{}^2=L_{2u}-(-1)^u2\,,\quad u\in\Z\,,
\end{equation}
and making use of the identity
\begin{equation}\label{equ.ewdkjs3}
L_v{}^2=L_{2v}+(-1)^v2\,,\quad v\in\Z\,,
\end{equation}
and finally setting $u=mk$, it is established that
\begin{equation}\label{equ.v1jp3hh}
25F_{mk}{}^4=L_{4mk}+(-1)^{mk-1}4L_{2mk}+6\,.
\end{equation}
By summing both sides of identity~\eqref{equ.v1jp3hh}, using Lemma~\ref{thm.hnf079l} to sum each of the first two terms on the right hand side, we have
\begin{equation}\label{equ.e697qms}
25\sum_{k = 1}^n {F_{mk}{}^4 }  = \frac{{(F_{2mn} L_{2mn + 2m}  + 4( - 1)^{mn - 1} L_m F_{mn}L_{mn + m} )}}{{F_{2m} }} + 6n\,,
\end{equation}
Using the identity~\eqref{equ.yb05ue2} we can write
\begin{equation}\label{equ.evmex9j}
\begin{split}
F_{2mn} L_{2mn + 2m}  &= F_{4mn + 2m}  - F_{2m}\\
&= F_{2mn + m} L_{2mn + m}  - F_{2m}
\end{split}
\end{equation}
and
\begin{equation}\label{equ.rg368e5}
\begin{split}
L_m F_{mn} L_{mn + m}  &= L_m (F_{2mn + m}  - ( - 1)^{mn} F_m )\\
&= L_m F_{2mn + m}  + ( - 1)^{mn - 1} F_{2m}\,.
\end{split}
\end{equation}
Substituting \eqref{equ.evmex9j} and \eqref{equ.rg368e5} into \eqref{equ.e697qms} proves Theorem~\ref{thm.rue8com}.
\end{proof}
\begin{cor}
If $n$ is an integer, then
\[
25\sum_{k=1}^nF_k{}^4=F_{2n+1}L_{n-1}L_{n+2}+6n+3\,.
\]
\end{cor}
\begin{proof}
From Theorem~\ref{thm.rue8com} we have
\begin{equation}
25\sum_{k=1}^nF_k{}^4=F_{2n+1}(L_{2n+1}+4(-1)^{n-1})+6n+3\,.
\end{equation}
From identity~\eqref{equ.ldun001} with $u=n+2$ and $v=n-1$ we have
\begin{equation}
L_{2n+1}+4(-1)^{n-1}=L_{2n+1}+(-1)^{n-1}L_3=L_{n-1}L_{n+2}\,,
\end{equation}
and the result follows.
\end{proof}
\begin{thm}\label{thm.qgnytde}
If $m$ is a non-zero integer and $n$ is any integer, then
\[
\sum_{k = 1}^n {L_{mk}{}^4 }  = \frac{{F_{2mn+m} (L_{2mn + m}  + 4( - 1)^{mn} L_m )}}{{F_{2m} }} + 6n-5\,.
\]
\end{thm}
\begin{proof}
The theorem is proved by summing both sides of the following identity,
\begin{equation}\label{equ.t6kphab}
L_{mk}{}^4=L_{4mk}-(-1)^{mk-1}4L_{2mk}+6\,,
\end{equation}
applying Lemma~\ref{thm.hnf079l} to sum each of the first two terms on the right hand side.
Identity~\eqref{equ.t6kphab} is obtained by squaring identity~\eqref{equ.ewdkjs3} and finally setting $v=mk$.
\end{proof}
\begin{cor}
If $n$ is an integer, then
\[
\sum_{k=1}^nL_k{}^4=5F_{2n+1}F_{n-1}F_{n+2}+6n-5\,.
\]
\end{cor}
\begin{proof}
From Theorem~\ref{thm.qgnytde} we have
\begin{equation}
\sum_{k=1}^nL_k{}^4=F_{2n+1}(L_{2n+1}-4(-1)^{n-1})+6n-5\,.
\end{equation}
From identity~\eqref{equ.sz2doji} with $u=n+2$ and $v=n-1$ we have
\begin{equation}
L_{2n+1}-4(-1)^{n-1}=L_{2n+1}-(-1)^{n-1}L_3=5F_{n-1}F_{n+2}\,,
\end{equation}
and the result follows.
\end{proof}
\subsection{Alternating sums}
\begin{thm}\label{thm.x8sbxd0}
If $m$ and $n$ are integers, then
\[
\sum_{k = 1}^n {( - 1)^{k - 1} F_{mk}{}^4 }  = \frac{{F_{mn} F_{mn + m} \left\{ (-1)^{n-1}L_m L_{mn} L_{mn + m}  + (-1)^{n(m-1)}4L_{2m} \right\}}}{{5L_m L_{2m} }}\,.
\]
\end{thm}
\begin{proof}
Multiplying through identity~\eqref{equ.v1jp3hh} by $(-1)^{k-1}$ and summing over $k$, we have the identity
\begin{equation}\label{equ.gf4l6f1}
\begin{split}
25\sum_{k = 1}^n {( - 1)^{k - 1} F_{mk}{}^4 }  &= \sum_{k = 1}^n {( - 1)^{k - 1} L_{4mk} }\\
&\qquad  + 4\sum_{k = 1}^n {( - 1)^{k(m - 1)} L_{2mk} }  + 3(( - 1)^{n - 1}  + 1)\,.
\end{split}
\end{equation}
When Lemma~\ref{thm.njcpb71} is used to evaluate the sums on the right hand side we have
\begin{equation}
\begin{split}
25\sum_{k = 1}^n {( - 1)^{k - 1} F_{mk}{}^4 }  &= \frac{{( - 1)^{n - 1} L_{4mn + 2m}  + L_{2m} }}{{L_{2m} }}\\
&\qquad + \frac{{4\left\{ {( - 1)^{n(m - 1)} L_{2mn + m}  - L_m } \right\}}}{{L_m }}\\
&\qquad + 3\left\{ {( - 1)^{n - 1}  + 1} \right\}\,,
\end{split}
\end{equation}
that is,
\begin{equation}\label{equ.rajm5ce}
\begin{split}
25\sum_{k = 1}^n {( - 1)^{k - 1} F_{mk}{}^4 }  &= \frac{{( - 1)^{n - 1} L_{4mn + 2m} }}{{L_{2m} }} + \frac{{4( - 1)^{n(m - 1)} L_{2mn + m} }}{{L_m }} + 3( - 1)^{n - 1}\\ 
&= \frac{{( - 1)^{n - 1} \left\{ {L_{4mn + 2m}  - L_{2m} } \right\}}}{{L_{2m} }}\\
&\qquad + \frac{{4( - 1)^{n(m - 1)} \left\{L_{2mn + m}  - ( - 1)^{mn} L_m\right\} }}{{L_m }}\,.
\end{split}
\end{equation}
Theorem~\ref{thm.x8sbxd0} then follows when the identities 
\begin{equation}\label{equ.sz2doji}
L_{u + v} - (-1)^vL_{u-v}=5F_vF_u  
\end{equation}
and
\begin{equation}\label{equ.yc4034r}
F_{2u}=F_uL_u
\end{equation}
are used to write the right hand side of~\eqref{equ.rajm5ce}.
\end{proof}
\begin{cor}
If $n$ is an integer, then
\[
\sum_{k=1}^n(-1)^{k-1}F_k{}^4=\frac{(-1)^{n-1}}3F_nF_{n+1}F_{n-2}F_{n+3}\,.
\]
\end{cor}
\begin{proof}
From Theorem~\ref{thm.x8sbxd0}
\begin{equation}
\sum_{k=1}^n(-1)^{k-1}F_k{}^4=\frac{F_nF_{n+1}(( - 1)^{n - 1}L_nL_{n+1}+L_2L_3)}{15}\,.
\end{equation}
From identity~\eqref{equ.sz2doji}
\begin{equation}\label{equ.s69j9wb}
L_nL_{n+1}=L_{2n+1}-(-1)^{n-1}\,,\quad L_2L_3=L_5+1\,.
\end{equation}
We therefore have
\[
\begin{split}
\sum_{k = 1}^n {( - 1)^{k - 1} F_k{}^4 }  &= \frac{{( - 1)^{n - 1} F_n F_{n + 1} (L_{2n + 1}  + ( - 1)^{n - 1} L_5 )}}{{15}}\\
&=\frac{{( - 1)^{n - 1} F_n F_{n + 1} (L_{2n + 1}  - ( - 1)^{n - 2} L_5 )}}{{15}}\\
&=\frac{(-1)^{n-1}}3F_nF_{n+1}F_{n-2}F_{n+3}\,,\quad\mbox{by identity~\eqref{equ.sz2doji}.}
\end{split}
\]

\end{proof}
\begin{thm}\label{thm.rthrjbn}
If $m$ and $n$ are integers, then
\[
\sum_{k = (1+(-1)^n)/2}^n {( - 1)^{k - 1} L_{mk}{}^4 }  = \frac{{( - 1)^{n - 1}5F_{mn} F_{mn + m} \left\{ L_m L_{mn} L_{mn + m}  + ( - 1)^{nm}4L_{2m} \right\}}}{{L_m L_{2m} }}\,.
\]
\end{thm}
\begin{proof}
Multiplying through identity~\eqref{equ.t6kphab} by $(-1)^{k-1}$ and summing over $k$, we have the identity
\begin{equation}\label{equ.gf4l6f1}
\begin{split}
\sum_{k = 1}^n {( - 1)^{k - 1} L_{mk}{}^4 }  &= \sum_{k = 1}^n {( - 1)^{k - 1} L_{4mk} }\\
&\qquad  - 4\sum_{k = 1}^n {( - 1)^{k(m - 1)} L_{2mk} }  + 3(( - 1)^{n - 1}  + 1)\,,
\end{split}
\end{equation}
which by the use of Lemma~\ref{thm.njcpb71} gives
\[
\sum_{k = 1}^n {( - 1)^{k - 1} L_{mk}{}^4 }  = \frac{{( - 1)^{n - 1} L_{4mn + 2m} }}{{L_{2m} }} - \frac{{4( - 1)^{n(m - 1)} L_{2mn + m} }}{{L_m }} + 3( - 1)^{n - 1}+8\,,
\]
so that if $n$ is even we have
\begin{equation}\label{equ.udj4uic}
\sum_{k = 1}^n {( - 1)^{k - 1} L_{mk}{}^4 }  =  - \frac{{(L_{4mn + 2m}  - L_{2m} )}}{{L_{2m} }} - \frac{{4(L_{2mn + m}  - L_m )}}{{L_m }}\,,
\end{equation}
while if $n$ is odd we have
\[
\begin{split}
\sum_{k = 1}^n {( - 1)^{k - 1} L_{mk}{}^4 }  &=  - \frac{{(L_{4mn + 2m}  - L_{2m} )}}{{L_{2m} }}\\
&\qquad - \frac{{4( - 1)^{m - 1} (L_{2mn + m}  - ( - 1)^m L_m )}}{{L_m }} + 16\,,
\end{split}
\]
that is,
\begin{equation}\label{equ.ddszwu4}
\begin{split}
\sum_{k = 0}^n {( - 1)^{k - 1} L_{mk}{}^4 }  &=  - \frac{{(L_{4mn + 2m}  - L_{2m} )}}{{L_{2m} }}\\
&\qquad - \frac{{4( - 1)^{m - 1} (L_{2mn + m}  - ( - 1)^m L_m )}}{{L_m }}\,.
\end{split}
\end{equation}
Using identities~\eqref{equ.sz2doji} and~\eqref{equ.yc4034r} to write the right side of identities~\eqref{equ.udj4uic} and \eqref{equ.ddszwu4} and combining the results we obtain the statement of Theorem~\ref{thm.rthrjbn}.
\end{proof}
\begin{cor}
If $n$ is an integer, then
\[
\sum_{k = (1 + ( - 1)^n )/2}^n {( - 1)^{k - 1} L_k^4 }  = ( - 1)^{n - 1} \frac{5}{3}F_n F_{n + 1} (L_{n - 2} L_{n + 3}  + ( - 1)^n 2)\,.
\]

\end{cor}

\end{document}